\documentclass[11pt]{amsart} \textwidth=14.5cm \oddsidemargin=1cm
\usepackage[rgb]{xcolor}
\usepackage{tikz}
\usepackage{verbatim}
\usepackage{amsmath}
\usepackage{amsxtra}
\usepackage{amscd}
\usepackage{amsthm}
\usepackage{amsfonts}
\usepackage{amssymb}
\usepackage{eucal}

\usepackage[linktocpage=true]{hyperref}
\usepackage[all, cmtip]{xy}
\usepackage{stmaryrd}
\usepackage{color}

\newtheorem{thm}{Theorem} [section]

\newtheorem{lem}[thm]{Lemma}
\newtheorem{prop}[thm]{Proposition}

\theoremstyle{definition}

\theoremstyle{remark}
\newtheorem{rem}[thm]{Remark}

\numberwithin{equation}{section}

\begin{document}

\newcommand{\thmref}[1]{Theorem~\ref{#1}}
\newcommand{\secref}[1]{Section~\ref{#1}}
\newcommand{\lemref}[1]{Lemma~\ref{#1}}
\newcommand{\propref}[1]{Proposition~\ref{#1}}
\newcommand{\corref}[1]{Corollary~\ref{#1}}
\newcommand{\remref}[1]{Remark~\ref{#1}}
\newcommand{\eqnref}[1]{(\ref{#1})}
\newcommand{\exref}[1]{Example~\ref{#1}}

\newcommand{\nc}{\newcommand}
 \nc{\Z}{{\mathbb Z}}
 \nc{\C}{{\mathbb C}}
 \nc{\N}{{\mathbb N}}
  \nc{\Q}{\mathbb{Q}}
 \nc{\la}{\lambda}
 \nc{\ep}{\epsilon}
 \nc{\n}{\mf n} 
 \nc{\La}{\Lambda}
  \nc{\V}{\mf V}
 \nc{\bi}{\bibitem}
 \nc{\E}{\mc E}
 \nc{\ba}{\tilde{\pa}}
 \nc{\half}{\frac{1}{2}}
 \nc{\hgt}{\text{ht}}
 \nc{\mc}{\mathcal}
 \nc{\mf}{\mathfrak} 
 \nc{\hf}{\frac{1}{2}}
 \nc{\hgl}{\widehat{\mathfrak{gl}}}
 \nc{\gl}{{\mathfrak{gl}}}
 \nc{\so}{{\mathfrak{so}}}
 \nc{\hz}{\hf+\Z}
 
\nc{\OO}{\mc{O}}
\nc{\Oi}{\mathcal O^{\text{int}}}

\nc{\ov}{\overline}
\nc{\ul}{\underline}
\nc{\wt}{\widetilde}
\nc{\I}{\mathbb{I}}
\nc{\X}{\mathbb{X}}
\nc{\Y}{\mathbb{Y}} 

\nc{\A}{\mathcal A}
\nc{\aA}{{}_\A}
\nc{\B}{\bold B}
\nc{\ff}{{\textbf f}}
\nc{\aff}{{}_\A\ff}
\nc{\Qq}{{\mathbb Q}(q)}
\nc{\dM}{{}^\omega M}
\nc{\dL}{{}^\omega L}
\nc{\aM}{{}_\A M}
\nc{\aL}{{}_\A L}
\nc{\daM}{{}^\omega_\A M}
\nc{\daL}{{}^\omega_\A L}
\nc{\htt}{\text{tr }}
\nc{\lwm}{{}_{\text{lw}}M}
\nc{\hwm}{{}^{\text{hw}}M}
\nc{\U}{\bold U}
\nc{\Udot}{\dot{\U}}
\nc{\UdotA}{{}_\A \dot{\U}}
\nc{\Uq}{{\mathcal U}_q}

 \nc{\be}{e}
 \nc{\bff}{f}
 \nc{\bk}{k}
 \nc{\bt}{t}
\nc{\id}{\text{id}}
\nc{\Ihf}{\I^\imath}
\nc{\one}{\bold{1}}

\newcommand{\blue}[1]{{\color{blue}#1}}
\newcommand{\red}[1]{{\color{red}#1}}
\newcommand{\green}[1]{{\color{green}#1}}
\newcommand{\white}[1]{{\color{white}#1}}

\title[Canonical bases in tensor products revisited ]
{Canonical bases in tensor products revisited}  
 
 \author[Huanchen Bao]{Huanchen Bao}
\address{Department of Mathematics, University of Virginia, Charlottesville, VA 22904}
\email{hb4tb@virginia.edu (Bao), \quad ww9c@virginia.edu (Wang)}

\author[Weiqiang Wang]{Weiqiang Wang}

\begin{abstract}
We construct canonical bases in  tensor products of several lowest and highest weight  integrable modules, 
generalizing Lusztig's  work. 
\end{abstract}


\maketitle

\let\thefootnote\relax\footnotetext{{\em 2010 Mathematics Subject Classification.} Primary 20G42. Secondary 17B37.}


\section{Introduction}

Lusztig \cite{Lu92} constructed the canonical basis in a tensor product 
of a lowest weight and a highest weight integrable module (
denoted by $\lwm \otimes \hwm$),
but not on $\hwm \otimes \hwm$ over a general quantum group $\U$. 
In \cite{Lu93}, he generalized this construction by defining the notion of based modules, but was only able to fully develop
this theory in finite type;  as a consequence, he constructed the canonical bases for tensor products of  several (finite-dimensional) modules for $\U$ of finite type.
These canonical bases have important applications to category $\mc O$, categorification, and quantum topology.

The goal of this paper is  
to provide a very simple {\em algebraic} construction of canonical bases in  tensor products \eqref{eq:Tlh} of 
several lowest weight integrable modules followed by  
highest weight integrable modules
over $\U$ of {\em Kac-Moody type} --- this settles a basic problem, open since Lusztig's work more than two decades ago.
We do so by extending the essential parts of   \cite[Chapter 27]{Lu93} 
on based modules  to Kac-Moody setting. 

 Zheng \cite{Z08} and Webster  \cite{We13} have categorified  tensor products of highest weight integrable modules over $\U$ of Kac-Moody type.
Moreover Webster \cite{We12} has further categorified the more general tensor products of the form \eqref{eq:Tlh}, 
building on  works of
Khovanov, Lauda, Rouquier, Vasserot, and Varagnolo \cite{KL09, R08, VV11}.  
The unavailability of an algebraic construction of tensor product canonical bases in full generality 
has been puzzling and hence our work helps to fill a gap in the program of categorification.
In general it is a difficult and deep problem to decide whether
the {\em can$\tiny \oplus$nical basis} (which is defined to be the basis of projective indecomposables) coincides with the canonical basis. 
In the setting of $\hwm \otimes \ldots \otimes \hwm$ over $\U$ of symmetric type, the classes of projective indecomposable modules
in Webster's category do coincide with the canonical basis constructed in this paper
(see Theorem~\ref{th:PIM=CB}; Webster \cite{We12}  proved this for finite $ADE$ types).
A can$\tiny \oplus$nical basis  
comes with  positivity 
but is extremely difficult to compute, while  an (algebraic) canonical basis  
is   
computable by the Gram-Schmidt algorithm.  
  
The quasi-$\mc R$-matrix $\Theta$, a variant of Drinfeld's 
universal $\mc R$-matrix \cite{Dr86}, was introduced by Lusztig to define a bar involution on tensor product modules. 
The key step in our approach is a simple proof that $\Theta$ preserves the $\Z[q,q^{-1}]$-forms of modules
such as  $\lwm \otimes M$ or  $M \otimes \hwm$ for any based module $M$.
Our argument bypasses the integrality issue of the quasi-$\mc R$-matrix $\Theta$ (which was only known in finite type
\cite[24.1.6]{Lu93}) and  simultaneously gets around the cyclicity of $\lwm \otimes \hwm$ used in \cite[23.3.6]{Lu93}.
In this way, we show that $\lwm \otimes M$ and  $M \otimes \hwm$ are based modules, and
this leads to the canonical bases of the tensor product modules  \eqref{eq:Tlh} inductively.

\section{Canonical bases in tensor products}

We will follow the notation and convention of the book \cite{Lu93} unless otherwise specified. 
Throughout we shall assume that the root system is $Y$-regular.
\subsection{Approximate cyclicity}

For $\la \in X$,
let $M(\la)$ be the  Verma module and $L(\la)$ be the highest weight simple module of highest weight $\la$
of a quantum group $\U$. 
We identify the underlying vector space for $M(\la)$ as $\ff$ \cite[1.2]{Lu93} with highest weight vector identified with $1\in \ff$.
Let $\B$ be the canonical basis of $\ff$. We identify $\ff$ with $\U^-$ via the isomorphism $\ff \rightarrow \U^-$, $b\mapsto b^-$, 
and denote by $\B^-$ the  canonical basis in $\U^-$.

A based $\U$-module in this paper is a $\U$-module which satisfies Conditions (a)-(d) in \cite[27.1.2]{Lu93} 
and is integrable (the integrability here replaces the finite-dimensionality condition in {\em loc. cit.}).
A basic example of based modules is $L(\la)$ for $\la \in X^+$ with its 
canonical basis of Lusztig and Kashiwara \cite{Lu90, Ka91, Lu93}. 
Let $(M,B)$ and $(M',B')$ be two based $\U$-modules  whose associated bar involutions will both be denoted by $\overline{\phantom{r}}$.
The  new (anti-linear) bar map $\Psi$ on $M\otimes M'$ given by $\Psi(m\otimes m') =\Theta (\bar{m} \otimes \bar{m}')$ makes sense, if 
$\Theta: M \otimes M' \rightarrow M \otimes M'$ is well defined 
(that is, $\Theta (m\otimes m')$ is always a finite sum for all $m\in M, m'\in M'$).

\begin{rem}
 \label{rem:well}
Since $\Theta$ lies in a completion of $\U^-\otimes \U^+$ by \cite[Theorem ~4.1.2]{Lu93},
the map $\Theta: M \otimes M' \rightarrow M \otimes M'$ is well defined if the following condition is satisfied:
$$
(\star)\;\; xm \otimes x'm'=0, \forall m\in M, m'\in M',  \forall x \in \U^-_{\nu}, x' \in \U^+_\nu \text{ with } \nu \text{ sufficiently large.}
$$ 
In particular, the condition ($\star$) is satisfied when $M_1 =\dL(\la)$ {\em or}  $M_2 =L(\la)$, for $\la \in X^+$. 
\end{rem}

Let $\A=\Z[q,q^{-1}]$. We have the $\A$-submodules $\aA\ff$ and $\aA\U^-$ in $\ff$ and $\U^-$ generated by $\B$
and $\B^-$, respectively \cite{Lu93}.  For a based module $(M, B)$, we let $\aM$ be the $\A$-submodule of $M$
generated by $B$. Recall \cite[Chapter 23]{Lu93} $\Udot$ is the modified quantum group containing various idempotents ${\bf 1}_\la$ for
$\la \in X$.

\begin{lem}
\label{lem:surjA1}Let $(M, B(M))$ be a based $\U$-module and let   $\la \in X$.
\begin{enumerate}
\item
For $b \in B(M)$, the $\Qq$-linear map 
$
\pi_b: \U^- {\bf 1}_{|b|+\la} \longrightarrow M \otimes M(\la),  \;
u \mapsto u(b \otimes 1),
$
restricts to an $\A$-linear map
$\pi_b: \aA \U^- {\bf 1}_{|b|+\la} \longrightarrow \aA M \otimes_\A \aM(\la).$

\item
We have $\sum_{b \in B(M)}  \pi_b (\aA \U^- {\bf 1}_{|b|+\la})= \aA M \otimes_\A \aM(\la)$.
\end{enumerate}
\end{lem}

\begin{proof}
Let $b' \in \B$. Using the comultiplication $\Delta$ in \cite[3.1.4]{Lu93}, we can write 
\begin{equation}
\label{eq:Db--}
\Delta ({b'}^-) =1\otimes {b'}^- + \sum c_{b_1, b_2} b_1^-  \otimes b_2^-
\in {}_\A \U^- {}_\A \U^0 \otimes  {}_\A \U^-,
\end{equation}
where the sum is taken over $b_1, b_2 \in \B$ such that $\htt |b_1| \leq \htt |b'|, \htt |b_2| <\htt |b'|$, and $c_{b_1, b_2} \in \A$; 
see \cite[1.1-1.2]{Lu93} for notations.
Then
\begin{align}
\pi_b ({b'}^- {\bf 1}_{|b| +\la}) = {b'}^- {\bf 1}_{|b| +\la} (b \otimes 1) 
 &=  b \otimes b' + \sum c_{b_1, b_2}  b_1^- b \otimes b_2.
 \label{eq:b+-2}
\end{align}
Part (1) follows. 

Note that $\aff$ has an increasing filtration 
$$
\A = \aff_{\le 0} \subseteq \aff_{\le 1} \subseteq \cdots \subseteq \aff_{\le N} \subseteq \cdots
$$
where $\aff_{\le N}$ denotes the $\A$-span of $\{\theta_{i_1}^{(a_1)}\ldots \theta_{i_n}^{(a_n)}   
| a_1+\ldots + a_n\le N, i_1, \ldots, i_n \in I\}.$
This induces an  increasing filtration on $\aM(\la)$. 

Set $Z := \sum_{b \in B(M)}  \pi_b ( \aA \U^- {\bf 1}_{|b| +\la})$. We have by (1) that $Z\subseteq \aM \otimes_\A \aM(\la)$.
To prove (2), it suffices to prove by induction on $N$ that $\aM \otimes_\A \aM(\la)_{\le N}   \subseteq Z$,
with the base case $N=0$ covered by definition.

Let $b' \in\aff_{\le N}$, for $N\ge 1$. Recalling notations from \eqref{eq:Db--},
we have $\sum c_{b_1, b_2} b_1^-  b \otimes b_2 \in   \aM \otimes_\A \aM(\la)_{\le N-1}$,
which lies in $Z$ by the inductive assumption. Since the left-hand side of \eqref{eq:b+-2} 
lies in $Z$ by Part~ (1), we have by \eqref{eq:b+-2} that $b \otimes b' \in Z$.
Letting $b' \in \aff_{\le N}$ and $b \in B(M)$ vary, we conclude that $\aM  \otimes_\A \aM(\la)_{\le N}   \subseteq Z$.  
Part (2)  is proved. 
\end{proof}

For $\la \in X^+$, we denote by $\eta_\la$ the image of $1$ under the projection $p_\la: M(\la) \rightarrow L(\la)$. 
Note that $p_\la$ restricts to $p_\la: \aM(\la) \rightarrow \aL(\la)$. The next lemma follows from Lemma~\ref{lem:surjA1}. 

\begin{lem}
\label{lem:surjA2}
Let $\la \in X^+$, and let $(M, B(M))$ be a based $\U$-module. 
\begin{enumerate}
\item
For $b \in B(M)$, the $\Qq$-linear map 
$
\pi_b: \U^- {\bf 1}_{|b|+\la} \longrightarrow M \otimes L(\la),  
\; u \mapsto u(b \otimes \eta_\la)$, restricts to an $\A$-linear map
$\pi_b: \aA \U^- {\bf 1}_{|b|+\la} \longrightarrow \aA M \otimes_\A \aL(\la).$

\item
We have $\sum_{b \in B(M)}  \pi_b (\aA \U^- {\bf 1}_{|b|+\la})= \aA M \otimes_\A \aL(\la)$.
\end{enumerate}
\end{lem}

The above lemmas provide us a key tool to approximate and get around the cyclicity of the tensor product 
of a lowest weight integrable module and a highest weight integrable module in \cite[23.3.6, 23.3.8]{Lu93}.

\subsection{Quasi-$\mc R$-matrix  and $\A$-forms}

The quasi-$\mc R$-matrix $\Theta$ induces a well-defined $\Qq$-linear map
$$
\Theta: M \otimes L(\la) \longrightarrow M \otimes L(\la),
$$ 
for $\la \in X^+$ and any weight module $M$; cf. \cite[24.1.1]{Lu93}.
The following is a generalization of \cite[Proposition~ 24.1.4, Corollary 24.1.5]{Lu93},
where Lusztig deals with the tensor product of a lowest weight  module and a highest weight module.

\begin{prop}
\label{prop:ThetaAM1}
Let $\la\in X^+$ and let $(M,B(M))$ be a based $\U$-module.
Then the $\Qq$-linear map $\Theta: M \otimes L(\la) \longrightarrow M \otimes L(\la)$ preserves the $\A$-submodule $\aM \otimes_\A \aL(\la)$. 
\end{prop}

\begin{proof}
As usual, we write $\overline{\phantom{u}}$ for $\overline{\phantom{r}} \otimes \overline{\phantom{r}}$ on $M \otimes L(\la)$, which clearly preserves
the $\A$-lattice $\aM \otimes_\A \aL (\la)$. 
Let $x\in \aM \otimes_\A \aL(\la)$. Then $x=\ov{x'}$ for some $x' \in \aM \otimes_\A \aL(\la)$. 
By Lemma~\ref{lem:surjA1}, we can write $x' = \sum_i \pi_{b_i} (u'_i)$ (a finite sum), for some $b_i\in B(M)$ and $u'_i \in \aA\U^- {\bf 1}_{|b_i|+\la}$. 
Since $\aA\U^- {\bf 1}_{|b_i|+\la}$ is preserved by the bar involution on $\Udot$, we have $u'_i =\bar{u}_i$ for some $u_i \in \aA\U^- {\bf 1}_{|b_i|+\la}$.
Hence 
\begin{equation}
\label{eq:xx'2}
x=\ov{x'} =\sum_i \ov{\bar{u}_i ( b_i \otimes \eta_\la)}.
\end{equation} 

We now recall a general property of the quasi-$\mc R$-matrix $\Theta$ \cite[Lemma 24.1.2]{Lu93}:
$$
u \Theta (m \otimes m') =\Theta (\overline{\bar{u}(\bar{m} \otimes \bar{m}')}),
$$
for $u \in \Udot, m \in M, m' \in L(\la)$.
Taking $m=b_i =\overline{b}_i$ and $m'=\eta_\la=\overline{\eta_\la}$, we have
\begin{equation}
 \label{eq:uTh2}
 u(b_i \otimes \eta_\la) =\Theta (\overline{\bar{u}(b_i \otimes \eta_\la)} ), \qquad \forall u\in \Udot,
\end{equation}
since  $\Theta (b_i \otimes \eta_\la) =b_i \otimes \eta_\la$ (which follows
from that $\Theta$ lies in a completion of $\U^-\otimes \U^+$ \cite[4.1.2]{Lu93}). 
By \eqref{eq:xx'2} and \eqref{eq:uTh2}, we have 
$$
\Theta(x) = \sum_i \Theta(\ov{\bar{u}_i (b_i \otimes \eta_\la)}) = \sum_i u_i(b_i \otimes \eta_\la) =
\sum_i \pi_{b_i} (u_i),
$$
where the latter lies in $\aM \otimes_\A \aL(\la)$ by Lemma~\ref{lem:surjA1}. 
The proposition is proved. 
\end{proof}

\begin{rem}
The same argument as above shows that $\Theta: M \otimes M(\la) \rightarrow M \otimes M(\la)$
preserves the $\A$-submodule $\aM \otimes \aM(\la)$, for each $\la \in X$.
\end{rem}


Recall from \cite[3.1.3]{Lu93} the automorphism $\omega$ of the $\Qq$-algebra $\U$. By twisting,
any $\U$-module $M$ gives rise to another $\U$-module ${}^\omega M$ with the same underlying vector space as $M$.
In particular, $\dM(\la)$ is the lowest weight Verma module and $\dL(\la)$ is the lowest weight simple module. 
The following dual statement to Proposition~\ref{prop:ThetaAM1} can be proved in a similar way
by first establishing dual versions of Lemmas~\ref{lem:surjA1} and \ref{lem:surjA2}.

\begin{prop}
\label{prop:ThetaAM2}
Let $\la\in X^+$ and let $M$ be a based $\U$-module.
Then the $\Qq$-linear map $\Theta: \dL(\la) \otimes M \longrightarrow \dL(\la) \otimes M$ preserves the $\A$-submodule $\daL(\la) \otimes_\A \aM$. 
\end{prop}

\subsection{Tensor product canonical bases}
 \label{sec:tensor}

Note that \cite[27.3.1]{Lu93} remains valid in Kac-Moody setting, 
conditional on that $\Theta: M \otimes M' \rightarrow M \otimes M'$ is well defined (cf.~ Remark~\ref{rem:well})
and that it preserves the $\A$-submodule $\aM \otimes \aM'$, 
for two based modules $(M,B), (M',B')$. 
Recall that $( (B, B'), <)$ is naturally a partially ordered set \cite[27.3.1]{Lu93}. 
A  highest (respectively, lowest) weight integrable module with its canonical basis is a based module \cite{Lu90, Ka91, Lu93}. 
We have the following generalization of \cite[Theorem 27.3.2]{Lu93} in the Kac-Moody setting.

\begin{thm}
\label{th:CB12}
Let $(M,B), (M',B')$ be two based modules, with either $M =\dL(\la)$ or $M' =L(\la)$ for $\la \in X^+$. 
Let $\mc L$ be the $\Z[q^{-1}]$-submodule of $M\otimes M'$ generated by $B\otimes B'$.

\begin{enumerate}
\item
For any $(b, b')\in B\times B'$, there is a unique element $b\diamondsuit b'\in \mc L$ such that 
$\Psi(b\diamondsuit b') =b\diamondsuit b'$ and $b\diamondsuit b' -b\otimes b' \in q^{-1} \mc L$.

\item
The element $b\diamondsuit b'$ is equal to $b\otimes b'$ plus a $q^{-1} \Z[q^{-1}]$-linear combination of
elements $b_2\otimes b_2'$ with $(b_2,b_2') \in B\times B'$ with $(b_2,b_2') <(b,b')$. 

\item
The elements $b\diamondsuit b'$ with $(b,b') \in B\times B'$ form a $\Qq$-basis of $M\otimes M'$,
an $\A$-basis of $\A \otimes_{\Z[q^{-1}]} \mc L$, and a $\Z[q^{-1}]$-basis of $\mc L$. 
\end{enumerate}
\end{thm}

\begin{proof}
By Remark~\ref{rem:well}, the map $\Theta: M \otimes M' \rightarrow M \otimes M'$ is well defined.
By Propositions~\ref{prop:ThetaAM1} and \ref{prop:ThetaAM2}, the map $\Theta$ preserves
the $\A$-submodule $\aM \otimes \aM'$.
Now the standard proof for  \cite[Theorem 27.3.2]{Lu93} goes through.
\end{proof}

\begin{rem}
\label{rem:opposite}
One can show that no variant of the quasi-$\mathcal R$-matrix exists
in a completion of $\U^-\otimes \U^+$ (instead of $\U^+ \otimes \U^-$ as in \cite{Lu93}) which intertwines $\Delta$ and $\overline{\Delta}$.
Hence we cannot define a bar involution on $L(\la) \otimes \dL(\mu)$ (or more general tensor products
of highest and lowest weight integrable modules in arbitrary order) via a quasi-$\mathcal R$-matrix.

If one uses the opposite coproduct from the one in \cite[3.1.4]{Lu93} then there is a version of quasi-$\mathcal R$-matrix
in a completion of $\U^-\otimes \U^+$ which goes with it. This gives rise to 
canonical bases of modules of the form $M \otimes \dL(\la)$ and $L(\la) \otimes M$ 
(with the $\U$-module structure  given by the opposite coproduct), for any based module $M$. 
But   this is just a reformulation of the constructions in this paper. 
\end{rem}

Let $(M,B), (M',B'),  (M'',B'')$ be based modules. We shall assume that 
the maps $\Theta$ on $M\otimes  M'$, $(M\otimes M') \otimes M''$, $M'\otimes M''$, and $M\otimes (M'\otimes M'')$
are all well defined (cf. Remark~\ref{rem:well}) and that they preserve the corresponding $\A$-submodules. 
Then in the same way as in \cite[27.3.6]{Lu93}, one shows that
$(M\otimes M') \otimes M''$ and $M\otimes (M'\otimes M'')$ are based modules and that
the associativity  $(M\otimes M') \otimes M'' \cong M\otimes (M'\otimes M'')$ (as based modules) holds, 
which allows us to write $M\otimes M' \otimes M''$.
This readily implies the analogous associativity result for more than three tensor factors. 

Let $r, \ell$ be integers with $0\le r \le \ell$.
Let $\la_1,   \ldots, \la_\ell \in X^+$ and  ${\underline{\la}} =(\la_1, \ldots, \la_\ell; r)$.
We shall consider the tensor product $\U$-module 
\begin{equation}
 \label{eq:Tlh}
\mathbb T^{\underline{\la}} =
\dL(\la_1) \otimes \ldots \otimes \dL(\la_r)\otimes L(\la_{r+1})\otimes \ldots \otimes L(\la_\ell).
\end{equation}
Let $\B_i$ denote  the  canonical basis of $\dL(\la_i)$ for each $1\le i\le r$ and of $L(\la_i)$ for $r<i\le \ell$.  
Let $\mc L^{\underline{\la}}$ be the $\Z[q^{-1}]$-submodule of $\mathbb T^{\underline{\la}}$
generated by $\B_1 \otimes \ldots \otimes \B_\ell$. 
By Remark~\ref{rem:well} and applying Theorem~\ref{th:CB12} inductively, we have established
the following generalization of Lusztig's result from finite type to Kac-Moody type. 

\begin{thm}
\label{th:CBlwhw}
\begin{enumerate}
\item
For any $(b_1, \ldots, b_\ell)\in \B_1 \times \ldots \times \B_\ell$, 
there is a unique element $b_1\diamondsuit \ldots \diamondsuit b_\ell \in \mc L^{\underline{\la}}$ such that 
$$
{ \quad }\quad
 \Psi(b_1\diamondsuit \ldots \diamondsuit b_\ell)
 =b_1\diamondsuit \ldots \diamondsuit b_\ell
\quad \text{and} \quad
  b_1\diamondsuit \ldots \diamondsuit b_\ell -b_1 \otimes \ldots \otimes b_\ell  \in q^{-1} \mc L^{\underline{\la}}.
 $$

\item
We have
$b_1\diamondsuit \ldots \diamondsuit b_\ell =b_1\otimes \ldots \otimes b_\ell 
+ \sum_{b_1',\ldots, b_\ell'} c_{b_1',\ldots, b_\ell'}^{b_1,\ldots, b_\ell} b_1'\otimes \ldots \otimes b_\ell',
$
with $(b_1',\ldots, b_\ell') \neq  (b_1,\ldots, b_\ell)$ and 
$c_{b_1',\ldots, b_\ell'}^{b_1,\ldots, b_\ell} \in q^{-1} \Z[q^{-1}]$. 

\item
The elements $b_1\diamondsuit \ldots \diamondsuit b_\ell$ with $(b_1, \ldots, b_\ell)\in \B_1 \times \ldots \times \B_\ell$
form a $\Qq$-basis of $\mathbb T^{\underline{\la}}$ in \eqref{eq:Tlh},
an $\A$-basis of $\A \otimes_{\Z[q^{-1}]} \mc L^{\underline{\la}}$, and a $\Z[q^{-1}]$-basis of $\mc L^{\underline{\la}}$. 

\item
The natural homomorphism $\mc L^{\underline{\la}} \cap \Psi(\mc L^{\underline{\la}})
\rightarrow \mc L^{\underline{\la}} /q^{-1} \mc L^{\underline{\la}}$
is an isomorphism. 
\end{enumerate}
\end{thm}
Following Lusztig, we call the basis in this theorem  the {\em canonical basis} of the tensor product module $\mathbb T^{\underline{\la}}$.
When $\ell=2$ and $r=1$, $\mathbb T^{\underline{\la}} =\dL(\la_1)\otimes L(\la_2)$, and the theorem reduces to 
\cite[Theorem 24.3.3]{Lu93}. Otherwise, the theorem is  new for $\U$ of infinite type with $\ell\ge 2$. 

\begin{rem}
 \label{rem:basedO}
Let $\mathcal O$ be the BGG category of $\U$-modules, and let $\Oi$ be the full subcategory of $\mathcal O$
of integrable $\U$-modules of finite-dimensional weight spaces. 
We consider the category $\Oi_{\text{b}}$ of based modules $(M,B)$ with $M \in \Oi$; a basic example is $L(\la)$ for $\la \in X^+$ with its 
canonical basis. 
The properties of based modules in \cite[27.1.1-27.1.8; 27.2.1-27.2.2]{Lu93} remain  valid in $\Oi_{\text{b}}$, 
where the argument in 27.1.8  ``by induction on $\dim M$" can be easily
modified to be ``by induction on $\dim M[\ge \la]^{hi}$ (where $M[\ge \la]^{hi}$ denotes 
the space of highest weight vectors of weight $\ge \la$)".
The counterparts of \cite[27.3.1-27.3.2, 27.3.6]{Lu93}  have already been addressed above, and
\cite[27.3.5]{Lu93} remains valid. 
\end{rem}

One noteworthy consequence of Remark~\ref{rem:basedO} (cf. \cite[27.1.7]{Lu93}) is the following
(which has applications in particular for $\U$ of affine type).

\begin{prop}
Let $\la_1, \ldots, \la_\ell \in X^+$. Let $\eta_i$ (and $\eta$, respectively)
denote the highest weight vector of $L(\la_i)$ for each $i$
(and of $L(\sum_{i=1}^\ell \la_i)$, respectively).
Then the (unique) homomorphism of $\U$-modules 
$$
\chi: L\Big(\sum_{i=1}^\ell \la_i \Big) \longrightarrow L(\la_1) \otimes \ldots \otimes L(\la_\ell),
\qquad \chi(\eta) =\eta_1\otimes \ldots \otimes \eta_\ell
$$ 
sends each canonical basis element to a canonical basis element.  
\end{prop}

\subsection{Relation to categorification}

Building on the remarkable works of Zheng, Khovanov, Lauda, Rouquier and his own earlier work \cite{Z08, KL09, R08, We13}, 
Webster \cite{We12}   categorified  
tensor products of lowest and highest weight integrable modules {\em exactly} of the form \eqref{eq:Tlh}
(in particular the tensor products are not in arbitrary order as stated in the earlier versions of \cite{We12}; also compare with our Remark~\ref{rem:opposite}). 
Denote by $K_q^0(\mathcal X^{\underline{\la}})$ the Grothendieck group of Webster's category $\mathcal X^{\underline{\la}}$ \cite[Definition 5.2]{We12}. 
The basis consisting of the classes of principal indecomposable modules in $K_q^0(\mathcal X^{\underline{\la}})$ 
(called an {\em orthodox basis} in \cite{We12})
is called a {\em can$\tiny \oplus$nical basis} here (and read as positively canonical or canonical plus basis). 

Note that for $\U$ of infinite type it is still an open question whether the projective indecomposable modules
in Webster's category provide a categorification of the canonical basis of $\dL(\la) \otimes L(\la')$, for $\la, \la'\in X^+$,
even though the latter has been constructed in \cite{Lu93}. 
Nevertheless, 
in the case of tensor products of highest weight integrable modules,
combining 
\cite[
Propositions~ 7.6, 7.7, Theorem~8.8]{We12} 
 (where the hard work was done based on 
earlier works of  Vasserot, Varagnolo and Rouquier \cite{VV11, R12})
with our Theorem~\ref{th:CBlwhw} provides the following
theorem (which was known \cite{We12} in finite $ADE$ type).
Note that we get the strongest result out of combining the categorification and algebraic approaches in (3) below.

\begin{thm}  [UVA]
\label{th:PIM=CB}
Assume that 
$\U$ is  of symmetric type (i.e., the generalized Cartan matrix is symmetric).
Let $\la_1,\ldots, \la_\ell \in X^+$ and $\underline{\la} =( \la_1,\ldots, \la_\ell; 0)$.
\begin{enumerate}
\item
Under Webster's identification $K^0_q(\mathcal X^{\underline{\la}}) \cong \mathbb T^{\underline{\la}}$, the 
can$\tiny \oplus$nical basis in $K^0_q(\mathcal X^{\underline{\la}})$ coincides with the canonical basis of $\mathbb T^{\underline{\la}}$
in Theorem~\ref{th:CBlwhw}.

\item
The matrix coefficients for the action of any canonical basis element of $\Udot$ on $\mathbb T^{\underline{\la}}$, 
with respect to the canonical basis of $\mathbb T^{\underline{\la}}$, always lie in $\Z_{\ge 0}[q,q^{-1}]$. 

\item
The coefficients in Theorem~\ref{th:CBlwhw}(2) satisfy that $c_{b_1',\ldots, b_\ell'}^{b_1,\ldots, b_\ell} \in q^{-1} \Z_{\ge 0} [q^{-1}]$. 
\end{enumerate}
\end{thm}

\begin{rem}
A counterpart of Theorem~\ref{th:PIM=CB}(1)(2) is also valid in the framework of Zheng \cite[Theorems~3.3.5, 3.3.6]{Z08}, who established this already
for finite $ADE$ type.
\end{rem}

We conjecture that the statements in Theorem~\ref{th:PIM=CB} hold also for $\underline{\la} =( \la_1,\ldots, \la_\ell; r)$, where $1<r<\ell$, 
when $\U$ is of symmetric and infinite type.

\begin{rem} \label{rem:affine}
Recall that $\Theta$ admits an integral expansion with respect to the canonical basis in finite type \cite[Corollary 24.1.6]{Lu93}. 
It is  natural to ask if such an integral expansion property of $\Theta$ holds in Kac-Moody setting in light of the integrality results 
in Propositions~\ref{prop:ThetaAM1} and \ref{prop:ThetaAM2}. However, as we learned from M.~Kashiwara, such an
integral expansion of $\Theta$ with respect to the canonical/global crystal basis no longer holds in affine type $A_1^{(1)}$. 
\end{rem}

\vspace{.2cm}

\noindent {\bf Acknowledgement.} 
The second author is partially supported by the NSF grants DMS-1101268 and DMS-1405131. 
We thank M.~Kashiwara for Remark~\ref{rem:affine}. 
We also thank the referee for helpful suggestions in improving the exposition. 

\red{
  
}


\end{document}